\documentclass[12pt]{amsart}

\usepackage{verbatim, amssymb, enumitem}
\usepackage[breaklinks=true,colorlinks=true,linkcolor=green,citecolor=red,urlcolor=blue]{hyperref}

\allowdisplaybreaks

\renewcommand \a{\alpha}
\renewcommand \b{\beta}

\newcommand \la{\lambda}
\newcommand \ve{\varepsilon}
\newcommand \id{\mathrm{id}}
\newcommand \br{\mathbb{R}}

\newcommand \rk{\operatorname{rk}}
\newcommand \Ker{\operatorname{Ker}}

\newcommand \Span{\operatorname{Span}}
\newcommand \Tr{\operatorname{Tr}}

\newcommand \cJ{\mathcal{J}}

\newcommand \cp{\mathcal{C}}
\newcommand \so{\mathfrak{so}}
\newcommand\ag{\mathfrak a}
\newcommand\g{\mathfrak g}

\newcommand\z{\mathfrak z}
\newcommand\m{\mathfrak m}

\newcommand \ad{\operatorname{ad}}

\newcommand \G{\Gamma}

\newcommand \ri{\mathrm{i}}

\newcommand \<{\langle}
\renewcommand \>{\rangle}
\newcommand \ip{\langle \cdot, \cdot \rangle}

\newcommand \mU{\mathcal{U}}

\theoremstyle{plain}
\newtheorem{theorem}{Theorem}
\newtheorem*{theorem*}{Theorem}
\newtheorem*{corollary*}{Corollary}

\newtheorem*{conj*}{Conjecture}
\newtheorem{lemma}{Lemma}

\newtheorem*{prop*}{Proposition}

\newtheorem*{ET}{Euler Theorem}

\theoremstyle{definition}

\newtheorem*{definition*}{Definition}

\theoremstyle{remark}
\newtheorem{remark}{Remark}

\makeatletter
\@namedef{subjclassname@2020}{%
  \textup{2020} Mathematics Subject Classification}
\makeatother

\begin{document}

\title{Stability of geodesic vectors in low-dimensional Lie algebras}

\author{An Ky Nguyen}
\address{\! Department of Mathematical and Physical Sciences, La Trobe University, Melbourne, Australia 3086}
\email{19042053@students.ltu.edu.au}

\author{Yuri Nikolayevsky}
\email{y.nikolayevsky@latrobe.edu.au}

\thanks{The second author was partially supported by ARC Discovery Grant DP210100951.}



\begin{abstract}
A naturally parameterised curve in a Lie group with a left invariant metric is a geodesic, if its tangent vector left-translated to the identity satisfies the Euler equation $\dot{Y}=\ad^t_YY$ on the Lie algebra $\g$ of $G$. Stationary points (equilibria) of the Euler equation are called geodesic vectors: the geodesic starting at the identity in the direction of a geodesic vector is a one-parameter subgroup of $G$. We give a complete classification of Lyapunov stable and unstable geodesic vectors for metric Lie algebras of dimension $3$ and for unimodular metric Lie algebras of dimension $4$.
\end{abstract}

\maketitle

\section{Introduction}
\label{s:intro}

Let $(\g,\ip)$ be a metric Lie algebra and $G$ its connected, simply connected Lie group equipped with the left-invariant Riemannian metric defined by $\ip$. Let $\gamma(t)$ be a smooth curve on $G$. Define its \emph{hodograph} to be the curve $Y(t)=(dL_{\gamma(t)^{-1}})\dot{\gamma}(t)$ in $\g$ (obtained by the left translation of the velocity vector to the identity).

The following characterisation of geodesics of left-invariant metrics is well known (e.g. \cite[Section~3]{A1}).
\begin{ET} A curve $\gamma(t)$ is an affinely parameterised geodesic on $G$ if and only if
\begin{equation}\label{eq:de}
\dot{Y}=\ad^t_YY,
\end{equation}
where $\ad^t_Y$ is the metric adjoint to $\ad_Y$.
\end{ET}
A point $X \in \g$ is a stationary point of equation~\eqref{eq:de}, if $\<X,[X,Y]\>=0$, for all $Y \in \g$. A stationary point $X \ne 0$ is called a \emph{geodesic vector}: the exponent of $\br X$ is a \emph{homogeneous geodesic}, a geodesic which is at the same a subgroup of $G$. Any metric Lie algebra (respectively, any metric Lie group) admits a geodesic vector (respectively, a homogeneous geodesic) \cite{Kai}. Similar result for homogeneous spaces has been established in \cite{KS}; for the current state of knowledge in the theory the reader is referred to the survey \cite{Dus} and the bibliography therein. Homogeneous spaces for which any nonzero tangent vector is geodesic are known as \emph{geodesic orbit spaces}. For a modern, comprehensive introduction to the theory the reader is referred to \cite{BN}.

From a dynamical point of view, geodesic vectors are equilibria of the equation~\eqref{eq:de}. The study of stability of such equilibria for various classes of metric Lie algebras is an ongoing project of D.\,Alekseevsky and the second author (a paper is currently under preparation). Recall that a stationary point $X$ is called (Lyapunov) \emph{stable}, if any solution $Y(t)$ of \eqref{eq:de} starting close to $X$ remains close to $X$ for all $t \ge 0$, and is called \emph{unstable} otherwise (note that any solution is defined for all $t \in \br$ as $G$ is complete). Clearly, the function $I_0(Y)=\|Y\|^2$ is a first integral of~\eqref{eq:de} (twice the energy), and so one can effectively study \eqref{eq:de} on the unit sphere of $(\g,\ip)$.

It is important to emphasise the connection with and the difference from the theory of relative equilibria and their stability (see \cite{A2} and the survey paper~\cite{Mon} and references therein). In Hamiltonian dynamics with a Lie group symmetry, a \emph{relative equilibrium} is a trajectory which is an orbit of a one-parameter subgroup. A relative equilibrium is called (Lyapunov) stable if all the trajectories starting close to it remain close for all positive times. In our context, relative equilibria are exponents of the geodesic vectors, that is, homogeneous geodesics on the metric Lie group $G$. Although a geodesic vector may be stable as a stationary point of~\eqref{eq:de}, the corresponding homogeneous geodesic does not have to be a stable relative equilibrium. The simplest nontrivial example is the 3-dimensional metric Lie algebra whose brackets relative to an orthonormal basis are given by $[e_1,e_2]=e_2,\, [e_1,e_3]=e_3$. The stationary point $-e_1$ is stable (see Theorem~\ref{t:3}\eqref{it:3X}\eqref{it:3s2J>0>s1J} or the last paragraph of Section~\ref{s:3}), but the corresponding metric Lie group is isometric to the hyperbolic space, and so any two geodesics starting at the same point diverge.

In this paper, we give a complete characterisation of stability of geodesic vectors in metric Lie algebras of dimension~3 and in unimodular metric Lie algebras of dimension~4. Denote $\cJ_X$ the linearisation of the right-hand side of \eqref{eq:de} at $X \in \g$, so that $\cJ_XY= \ad^t_XY + \ad^t_YX$ (see Section~\ref{ss:lin} for more details). Let $\sigma_k(\cJ_X), \; k=1,\dots,n$, where $n =\dim \g$, be the $k$-th symmetric function of the eigenvalues of $\cJ_X$ (so that $\det(\cJ_X - \la \, \id_\g)=(-\la)^n+\sum_{k=1}^{n} (-\la)^{n-k} \sigma_k(\cJ_X)$).

Our main results are as follows.
\begin{theorem} \label{t:3}
  Let $(\g,\ip)$ be a metric Lie algebra of dimension $3$.
  \begin{enumerate}[label=\emph{(\Alph*)},ref=\Alph*]
    \item \label{it:3X}
    A stationary point $X \in \g$ of equation~\eqref{eq:de} is stable if and only if one of the following conditions is satisfied:
        \begin{enumerate}[label=\emph{(\roman*)},ref=\roman*]
            \item \label{it:3J0}
            $\cJ_X=0$;

            \item \label{it:3s2J>0>s1J}
            $\sigma_2(\cJ_X) \ge 0 \ge \sigma_1(\cJ_X)$, with at least one of the two inequalities being strict.
        \end{enumerate}

    \item \label{it:3exist}
    Equation~\eqref{eq:de} always has a nonzero stable stationary point, and has an unstable stationary point unless either $\g$ is abelian or $\g = \so(3)$ and $\ip$ is a bi-invariant inner product.
  \end{enumerate}
\end{theorem}

\begin{theorem} \label{t:4uni}
  Let $(\g,\ip)$ be a unimodular, metric Lie algebra of dimension $4$, and let $\z \subset \g$ be its centre.
  \begin{enumerate}[label=\emph{(\Alph*)},ref=\Alph*]
    \item \label{it:X}
    Let $X \in \g$ be a stationary point of equation~\eqref{eq:de}.
    \begin{enumerate}[label=\emph{(\alph*)},ref=\alph*]
        \item \label{it:tz>1}
        If $\dim \z > 1$, then $X$ is stable if and only if $X \in \z$.

        \item \label{it:tz=1}
        If $\dim \z = 1$, then $X$ is stable if and only if one of the following conditions is satisfied:
        \begin{enumerate}[label=\emph{(\roman*)},ref=\roman*]
            \item \label{it:J0}
            $\cJ_X=0$;

            \item \label{it:s2J>0}
            $\sigma_2(\cJ_X)>0$;

            \item \label{it:Jnilp}
            $\rk \cJ_X = 1$ and $\<X,Z\>\Tr (\cJ_X \cJ_Z) < 0$, where $Z \in \z \setminus\{0\}$.
        \end{enumerate}

        \item \label{it:tz=0}
        If $\z$ is trivial, then $\g$ has an abelian ideal $\ag$ of dimension $3$. Take an arbitrary $e \in \g \setminus \ag$ and denote $A=(\ad_e)_{|\ag}$. Then $X$ is stable if and only if $X \in \ag$ and either $X = 0$ or in the expansion $\|\exp(sA^t)X\|^2=\|X\|^2 + a_k s^k + o(s^k)$, where $k > 0$ and $a_k \ne 0$, we have $k$ even and $a_k > 0$.
    \end{enumerate}

    \item \label{it:exist}
    Equation~\eqref{eq:de} always has a nonzero stable stationary point, and has an unstable stationary point unless either $\g$ is abelian or $\g = \so(3) \oplus \br$ and $\ip$ is a bi-invariant inner product.
  \end{enumerate}
\end{theorem}
The sets of stationary points of \eqref{eq:de} for different cases of Theorem~\ref{t:3} and Theorem~\ref{t:4uni} are given in Sections~\ref{s:3} and \ref{s:4uni} respectively (see also \cite{Mar} in the 3-dimensional case).

\begin{remark} \label{r:sigmas}
  It is not hard to give conditions for $\sigma_k(\cJ_X)$ which are necessary for stability (essentially expressing the fact that $\cJ_X$ has no eigenvalues with positive real part); which of them are also sufficient is a much more delicate question. For example, a stationary point of a unimodular algebra may only be stable if all eigenvalues of $\cJ_X$ lie on the imaginary axis (Lemma~\ref{l:lin}\eqref{it:trtr}), and so we can never have exponential stability (and in fact, even asymptotic stability in our cases).
\end{remark}

\begin{remark} \label{r:4tz0}
  Regarding Theorem~\ref{t:4uni}\eqref{it:X}\eqref{it:tz=0}, note that the fact that a unimodular centreless 4-dimensional Lie algebra has a codimension one abelian ideal is well known (see the beginning of Section~\ref{s:4uni}). We also note that the condition on the expansion of $\|\exp(sA^t)X\|^2$ is in fact \emph{finite}: in the worst possible scenario, one needs to compute the first seven terms (up to $s^6$), see Remark~\ref{r:A^n}.
\end{remark}

The paper is partially based on the results of the thesis \cite{Ngu} which the first author carried out under the supervision of Grant Cairns and the second author. The authors would like to thank Grant Cairns for his valuable comments and discussions.

\section{Preliminaries}
\label{s:pre}

Let $(\g,\ip)$ be a metric Lie algebra, and let $\z \subset \g$ be its centre.

\subsection{Linearisation}
\label{ss:lin}

The linearisation $\cJ_X$ of the right-hand side of~\eqref{eq:de} at a point $X \in \g$ is given by $\cJ_XY=\ad_X^tY+\ad_Y^tX$, so that
\begin{equation}\label{eq:cJgen}
\<\cJ_XY,Z\>=\<Y,[X,Z]\>+\<X,[Y,Z]\>,
\end{equation}
for $X, Y, Z \in \g$. In the following lemma we collect some elementary, but useful, facts.

\newpage

\begin{lemma} \label{l:lin}
{\ }

\begin{enumerate}[label=\emph{(\alph*)},ref=\alph*]
  \item \label{it:trtr}
  For all $X \in \g$ we have $\Tr \cJ_X=\Tr \ad_X$. If $X$ is a stable stationary point of~\eqref{eq:de}, we must have $\Tr \ad_X \le 0$, and if, in addition, $\g$ is unimodular, all the eigenvalues of $\cJ_X$ must lie on the imaginary axis.

  \item \label{it:cJskew}
  If $\cJ_X$ is skew-symmetric or $\ad_X$ is skew-symmetric \emph{(}in particular, if $\cJ_X=0$ or if $X \in \z$\emph{)}, then $X$ is a stable stationary point of~\eqref{eq:de}.

  \item \label{it:cJsing}
  If $X$ is a stationary point of~\eqref{eq:de}, then $\cJ_XX=\cJ_X^tX=0$, so that $X^\perp$ is an invariant subspace of $\cJ_X$.

  \item \label{it:sigmas4}
  Suppose $\g$ is unimodular and $\dim \g = 4$. If $X$ is a stable stationary point of~\eqref{eq:de}, then $\rk \cJ_X \le 2$ and $\sigma_2(\cJ_X) \ge 0$. \end{enumerate}
\end{lemma}
\begin{proof}
  The first statement of assertion~\eqref{it:trtr} follows from the fact that, by \eqref{eq:cJgen}, $\<\cJ_XY,Y\>=\<[X,Y],Y\>$, for all $X,Y \in \g$. If for a stationary point $X$ we have either $\Tr \ad_X > 0$ or $\Tr \ad_X=0$ and the real part of at least one eigenvalue of $\cJ_X$ is nonzero, than $\cJ_X$ has an eigenvalue with a positive real part and hence $X$ is unstable.

  For assertion~\eqref{it:cJskew} we first note that if $\ad_X$ is skew-symmetric (in particular, if $X \in \z$), then $\cJ_X$ is skew-symmetric by \eqref{eq:cJgen}, and that if $\cJ_X$ is skew-symmetric, then $X$ is a stationary point by \eqref{eq:cJgen}. Moreover, in the latter case, the function $I_1(Y)=\<X,Y\>, \; Y \in \g$, is a first integral of~\eqref{eq:de}, as also is the function $Y \mapsto \|X\|^2+I_0(Y)-2I_1(Y)=\|Y-X\|^2$ and so $X$ is stable.

  Assertion~\eqref{it:cJsing} follows from \eqref{eq:cJgen}.

  Assertion~\eqref{it:sigmas4} follows from assertions~\eqref{it:trtr} and~\eqref{it:cJsing}: the restriction of $\cJ_X$ to $X^\perp$ must either be nilpotent, or have three simple eigenvalues $0, \pm \omega \ri$, where $\omega \ne 0$.
\end{proof}

\begin{remark} \label{r:dim2}
  If $\g$ is abelian, then all $X \in \g$ are stationary points and are stable (the right-hand side of \eqref{eq:de} is zero). If $\dim \g = 2$ and $\g$ is not abelian, we can choose an orthonormal basis $\{e_1,e_2\}$ for $(\g,\ip)$ relative to which $[e_1,e_2]=\alpha e_2, \, \a > 0$. Then \eqref{eq:de} takes the form $\dot{y}_1=-\a y_2^2, \, \dot{y}_2=\a y_1y_2$, for $Y=(y_1,y_2)^t$. A point $X=(x_1,x_2)^t$ is stationary when $x_2=0$, and it is easy to see that a stationary point $X=(x_1,0)^t$ is stable if and only if $x_1 \le 0$.
\end{remark}

\begin{remark}\label{r:normal}
  For the algebras in Theorems~\ref{t:3} and~\ref{t:4uni} there always exist nonzero stable points. It would be interesting to know if this is so for any metric Lie algebra. As to unstable points, they clearly always exist if the algebra is non-unimodular: if $X \ne 0$ is orthogonal to the unimodular ideal, it must be stationary, and then one of $\Tr \ad_{\pm X}$ is positive, hence the corresponding point is unstable by Lemma~\ref{l:lin}\eqref{it:trtr}. One obvious case when there are no unstable points is when all $X \in \g$ are stationary, that is, when $\ad_X$ is skew-symmetric for all $X \in \g$ (from Theorems~\ref{t:3} and~\ref{t:4uni} we see that this is the only possibility when $\dim \g \le 4$). In this case, $G$ (more precisely, $G/\{e\}$) is a geodesic orbit space, and it is well known that it is isometric to the Riemannian product of simple compact groups with bi-invariant metrics and the Euclidean space (recall that $G$ is simply connected).
\end{remark}

\subsection{Invariant submanifolds and Lyapunov functions}
\label{ss:Lyapunov}

The main tool which we use for proving stability/instability (apart from the linearisation $\cJ_X$) is invariant submanifolds (foliations). These are submanifolds consisting of trajectories of~\eqref{eq:de} (tangent to the vector field $Y \mapsto \ad_Y^tY$). There are several sources of invariant submanifolds. First of all, we have submanifolds defined by first integrals. There is a wealth of results in the literature on integrability of the geodesic flow on Lie groups and homogeneous spaces. We note that for a semisimple $\g$, we always have a family of first integrals constructed from Casimir operators (but note, in our cases, $\g$ is rarely semisimple); in the nilpotent case, see a recent paper~\cite{KOR} and the bibliography therein. Another source of invariant submanifold is \emph{coadjoint orbits}. In the presence of inner product, we can consider them as lying in $\g$: these are the integral submanifolds of the distribution $Y \mapsto \ad_\g^tY$. A simple criterion of stability based on the coadjoint orbits is given in~\cite[Th\'{e}or\`{e}me~4]{A1}. Note that another class of invariant manifolds, central manifolds, has limited applications for us, at least in the unimodular case, as by Lemma~\ref{l:lin}\eqref{it:trtr}, a necessary condition for stability of a point $X \in \g$ is that the central manifold is the whole of $\g$.

A simple argument which we will frequently employ is that if we have a collection of local first integrals $I_j$ in a neighbourhood of a stationary point $X$, such that $X$ is locally the only common point of their level sets passing through $X$, then $X$ is stable (this follows from the fact that $V(Y)=\sum_j (I_j(Y)-I_j(X))^2$ is a Lyapunov function).

\section{Proof of Theorem~\ref{t:3}}
\label{s:3}

We separately consider the cases when $\g$ is unimodular and non-unimodular.

\medskip

Let $(\g,\ip)$ be a three-dimensional, \emph{unimodular} metric Lie algebra. Then for any $X \in \g$ we have $\sigma_1(\cJ_X)=0$ by Lemma~\ref{l:lin}\eqref{it:trtr}, and so the condition of stability in assertion~\eqref{it:3X} is that a stationary point $X$ is stable if and only if either $\cJ_X=0$ or $\sigma_2(\cJ_X)>0$.

By \cite[Lemma~4.1]{Mil} we can choose an orthonormal basis $\{e_1,e_2,e_3\}$ for $(\g,\ip)$ relative to which the Lie brackets are given by $[e_i,e_j]=\la_k e_k$, where $(i,j,k)$ is a cyclic permutation of $(1,2,3)$, and where we can assume that $\la_1 \ge \la_2 \ge \la_3$ (changing the sign of $e_1$ if necessary). Equation~\eqref{eq:de} takes the form $\dot{y}_i=(\la_j-\la_k)y_jy_k$ for $Y=(y_1,y_2,y_3)^t$, where $(i,j,k)$ is a cyclic permutation of $(1,2,3)$.

If $\la_1 = \la_2 = \la_3$, then either $\g$ is abelian, or $(\g,\ip)$ is $\so(3)$ with a bi-invariant inner product; moreover, all the points of $\g$ are stable and stationary. Otherwise, there is a first integral $I_1(Y)=\sum_{i=1}^{3} \la_i y_i^2$ functionally independent of $I_0$. If $\la_1 > \la_2 > \la_3$, then the set of stationary points is the union of the three coordinate axes. Considering the intersections of the level sets of $I_0$ and $I_1$ (or computing the eigenvalues of $\cJ_X$) we find that a stationary point $X$ is unstable if $X=(0,x_2,0)^t, \, x_2 \ne 0$, and is stable otherwise. If $\la_1 > \la_2 = \la_3$, then the set of stationary points is the union of the line $x_2=x_3=0$ and the plane $x_1=0$. The trajectories of \eqref{eq:de} are circles lying in the planes $y_1=c \ne 0$ with the centre on the $y_1$-axis, and so the stationary points on the line $x_2=x_3=0$ are stable, and the points on the plane $x_1=0$ other than the origin are unstable. Similarly, if $\la_1 = \la_2 > \la_3$, then the set of stationary points is the union of the line $x_1=x_2=0$ consisting of stable points and the plane $x_3=0$ all of whose points other than the origin are unstable. This proves assertion~\eqref{it:3exist} in the unimodular case. Moreover, computing $\cJ_X$ we find that $\cJ_X = 0$ if and only if either $X=0$ or $\la_1 = \la_2 = \la_3$, and that $\sigma_2(\cJ_X)= x_1^2(\la_3-\la_1)(\la_2-\la_1) + x_2^2(\la_1-\la_2)(\la_3-\la_2) + x_3^2(\la_2-\la_3)(\la_1-\la_3)$ which is positive at a nonzero stationary point $X$ exactly when $X$ is stable; this proves assertion~\eqref{it:3X} in the unimodular case.

\medskip

Let $(\g,\ip)$ be a three-dimensional, \emph{non-unimodular} metric Lie algebra. The unimodular ideal $\ag \subset \g$ is two-dimensional and abelian. We choose an orthonormal basis $\{e_i\}$ for $(\g,\ip)$ in such a way that $e_1 \perp \ag$, and denote $A= (\ad_{e_1})_{|\ag}$. Note that $\Tr A \ne 0$.

We start with proving assertion~\eqref{it:3X}. Equation \eqref{eq:de} takes the form
\begin{equation}\label{eq:EA3nuni}
  \dot{y}_1= -\<Ay,y\>,\qquad \dot{y}=y_1 A^t y,
\end{equation}
for $Y=y_1e_1+y \in \g$, where $y \in \ag$. The set of stationary points is the union of the subspace $\Span(e_1, \Ker A^t)$ and the cone in $\ag$ given by $\<Ax,x\>=0$ (for $X=x \in \ag$). We have $\cJ_X= \left(\begin{smallmatrix} 0 & -x^t(A+A^t) \\ A^tx & x_1A^t \end{smallmatrix}\right)$ at a stationary point $X=x_1e_1+x, \, x \in \ag$. By Lemma~\ref{l:lin}\eqref{it:cJskew} we can assume that $\cJ_X \ne 0$ (and in particular, that $X \ne 0$). Furthermore, by Lemma~\ref{l:lin}\eqref{it:cJsing}, $\cJ_X$ is singular. If $\sigma_1(\cJ_X) > 0$ or $\sigma_2(\cJ_X) < 0$, then $\cJ_X$ has an eigenvalue with positive real part, and so $X$ is unstable. If $\sigma_1(\cJ_X) < 0 < \sigma_2(\cJ_X)$, then both eigenvalues of the restriction of $\cJ_X$ to the subspace tangent at $X$ to the level surface of $I_0$ have negative real part, and so $X$ is stable. To prove assertion~\eqref{it:3X}, it therefore remains to consider the cases when $\cJ_X \ne 0$ and either $\sigma_1(\cJ_X) < 0 = \sigma_2(\cJ_X)$ or $\sigma_1(\cJ_X) = 0 \le \sigma_2(\cJ_X)$. Note that we have $\sigma_1(\cJ_X)=\Tr \cJ_X = x_1 \Tr A$ and $\sigma_2(\cJ_X) = \<A(A+A^t)x,x\>+x_1^2 \det A$.

We first suppose that the stationary point $X$ is such that $A^tx \ne 0$. Then $x_1=0$ and $\sigma_1(\cJ_X)= 0$ and $\sigma_2(\cJ_X) = \<A(A+A^t)x,x\>$. The trajectory of the solution of~\eqref{eq:EA3nuni} with the initial condition $Y(0)=y_{1}^0 e_1 + y^0, \, y^0 \in \ag \setminus \{0\}$, lies on the cylindrical surface $\cp(Y(0))$ given by $(w,s) \mapsto we_1 + \exp(sA^t)y^0$ (the coadjoint orbit of $Y(0)$; see Section~\ref{ss:Lyapunov}). Taking $Y(0)=X$ and restricting the first integral $I_0(Y)=\|Y\|^2$ to the surface $\cp(X)$ we obtain the function $F(w,s)= w^2 + \phi_x(s)$, where $\phi_x(s)= \|\exp(sA^t)x\|^2$. We have the expansion $\phi_x(s) = \|x\|^2 + \sigma_2(\cJ_X) s^2 + \frac13\<A^2(A+3A^t)x,x\>s^3 + o(s^3)$. If $\sigma_2(\cJ_X) > 0$, then the equation $F(w,s)=F(0,0)$ has locally only one solution, $(w,s)=(0,0)$, and so the level hypersurface $I_0(Y)=I_0(X)$ has only one point in common with the cylinder $\cp(X)$ in a neighbourhood of $X$ which implies that $X$ is stable (see Section~\ref{ss:Lyapunov}). Suppose $\sigma_2(\cJ_X) = 0$. Choose a basis for $\ag$ in such a way that $x=\a e_3, \, \a \ne 0$. As $\<Ax,x\>=0, \, A^tx \ne 0$ and $\Tr A \ne 0$, we obtain $A= \left(\begin{smallmatrix} a & b \\ c & 0 \end{smallmatrix}\right)$, with $a, c \ne 0$. Then $\sigma_2(\cJ_X) = \<A(A+A^t)x,x\>= c(b+c)\a^2$, and so $b=-c$, and then $\<A^2(A+3A^t)x,x\>=2ac^2 \a^2 \ne 0$, so that $\phi_x(s) = \|x\|^2 + \frac23 ac^2 \a^2 s^3 + o(s^3)$. Then the solution to $F(w,s)=F(0,0)$ is locally given by $w=\pm \sqrt{(-\frac23 ac^2 \a^2) s^3} + \dots$ which, outside the point $(w,s)=(0,0)$, is the union of two open analytic curves $\gamma_j, \, j=1, 2$, whose closures contain $(0,0)$. Note that locally these curves contain no stationary points of \eqref{eq:EA3nuni}, as for all of them, we locally have $y_1 = w \ne 0$, and $y$ is close to $x$, so that $A^ty \ne 0$. Then each of $\gamma_1, \gamma_2$ is a trajectory of a solution of \eqref{eq:EA3nuni} which tends to $X$ for either $t \to \infty$ or $t \to -\infty$. But both the system~\eqref{eq:EA3nuni} and the stationary point $X$ are invariant with respect to the change of variables $(y_1,t) \leftrightarrow (-y_1,-t)$, and so one of these trajectories tends to $X$ when $t \to \infty$, and another one, when $t \to -\infty$. This implies that $X$ is unstable.

We now suppose that the stationary point $X$ (with $\cJ_X \ne 0$) is such that $A^tx = 0$. Then $\sigma_1(\cJ_X) = x_1 \Tr A$ and $\sigma_2(\cJ_X) = x_1^2 \det A$ (and from the above argument the only remaining cases are either $\sigma_1(\cJ_X) \le 0 = \sigma_2(\cJ_X)$, or $\sigma_1(\cJ_X) = 0 < \sigma_2(\cJ_X)$). If $\sigma_2(\cJ_X) > 0$, then $x_1 \ne 0$, and so $\sigma_1(\cJ_X) \ne 0$; we can therefore assume that $\sigma_2(\cJ_X) = 0$. If $\det A \ne 0$, then $x=0$ and $x_1=0$, a contradiction. So $\det A = 0$ and we can choose a basis for $\ag$ in such a way that $Ae_3=0$. Then $A= \left(\begin{smallmatrix} a & 0 \\ b & 0 \end{smallmatrix}\right), \, a \ne 0$, and $\cJ_X = (-x_2,x_1,0)^t(0,a,b)$. Equation~\eqref{eq:EA3nuni} has a first integral $I_1(Y)=y_3$, and the trajectories of non-stationary solutions lie on the circles; the tangent line to the circle $I_0(Y)=I_0(X), \, I_1(Y)=I_1(X)$ is spanned by the vector $(-x_2,x_1,0)^t$ (which is nonzero as $\cJ_X \ne 0$), and $\cJ_X (-x_2,x_1,0)^t= ax_1 (-x_2,x_1,0)^t= \sigma_1(\cJ_X) (-x_2,x_1,0)^t$. So if $\sigma_1(\cJ_X) < 0$, the point $X$ is stable. If $\sigma_1(\cJ_X) = 0$, then $x_1=0$ and so $X$ is the only stationary point on the circle $I_0(Y)=I_0(X), \, I_1(Y)=I_1(X)$, which implies that $X$ is unstable. This completes the proof of assertion~\eqref{it:3X} in the non-unimodular case.

For assertion~\eqref{it:3exist} in the non-unimodular case, we first note that the stationary point $X=(\Tr A) e_1$ is unstable, as $\sigma_1(\cJ_X)= (\Tr A)^2>0$ (note that $\g$ cannot be abelian or isomorphic to $\so(3)$). To prove the existence of nonzero stationary points, we consider three cases. If $\det A = 0$, then a nonzero vector from the kernel of $A$ belongs to the centre of $\g$ and so is stable by Lemma~\ref{l:lin}\eqref{it:cJskew}. If $\det A > 0$, then the stationary point $X=(-\Tr A) e_1$ is stable by assertion~\eqref{it:3X}\eqref{it:3s2J>0>s1J}, as $\sigma_1(\cJ_X)= -(\Tr A)^2<0$ and $\sigma_2(\cJ_X)= (\Tr A)^2 \det A > 0$. Suppose $\det A < 0$. Let $A=S+K$, where $S$ is symmetric and $K$ skew-symmetric. Note that the quadratic form $\<Ax,x\>=\<Sx,x\>$ on $\ag$ takes values of both signs, and so is indefinite. Let $x_{(1)}, x_{(2)} \in \ag$ be two linearly independent vectors satisfying $\<Ax_{(j)},x_{(j)}\>=0,\, j=1,2$. Then each of the points $X=x_{(j)}$ is stationary, with $\sigma_1(\cJ_X)=0$, and with $\sigma_2(\cJ_X)=\<S^2x_{(j)},x_{(j)}\>- 2 \<Sx_{(j)},Kx_{(j)}\>$. If $K = 0$, then $\sigma_2(\cJ_X) > 0$, and so both points $X=x_{(1)}$ and $X=x_{(2)}$ are stable by assertion~\eqref{it:3X}. If $K \ne 0$, then $K^2=\a \, \id_\ag$ for some $\a < 0$, and so from the fact that $\<Sx_{(j)},x_{(j)}\>=0$, we obtain $Sx_{(j)}=\la_j Kx_{(j)}$, for some $\la_j \in \br, \, j=1,2$. Note that $\la_1+\la_2=\Tr(\a^{-1}KS)=0$ and that $\la_1, \la_2 \ne 0$ as $S \ne 0$. Then for $X=x_{(j)}$ with $\la_j < 0$ we have $\sigma_2(\cJ_X) \ge -2\<Sx_{(j)},Kx_{(j)}\> = 2\a\la_j \|x_{(j)}\|^2 > 0$, and so $X$ is stable by assertion~\eqref{it:3X}\eqref{it:3s2J>0>s1J}. This proves assertion~\eqref{it:3exist} in the non-unimodular case, and completes the proof of Theorem~\ref{t:3}.

\section{Proof of Theorem~\ref{t:4uni}}
\label{s:4uni}

Let $\g$ be a unimodular Lie algebra of dimension $4$. We claim that either the centre $\z$ of $\g$ is nontrivial, or $\g$ contains a three-dimensional abelian ideal. To see this, one may either inspect the classification in \cite{Mub}, or use the following argument. Suppose $\z=0$. Then the Levi subalgebra must be trivial (as otherwise it is three-dimensional, and then the radical is the centre), and so $\g$ is solvable. The derived algebra $\g'$ is nilpotent and nontrivial (as $\g$ is not abelian). It cannot be $1$-dimensional, as then $\g' \subset \z$, by unimodularity. If $\dim \g'=2$, then $\g'$ is abelian. If $L$ is its (linear) complement, the operators $(\ad_X)_{|\g'}, \; X \in L$, commute and have trace zero, so for some nonzero $X \in L$ we have $[X,\g']=0$, and then $\Span(\g',X)$ is an abelian ideal. If $\dim \g' = 3$ and $\g'$ is not abelian, it must be the Heisenberg algebra. But the kernel of any unimodular derivation of the Heisenberg algebra contains its centre which implies that $\z$ is nontrivial.

The proof goes as follows. Depending on $d_\z=\dim \z$, we consider each of the three cases \eqref{it:tz>1}, \eqref{it:tz=1} and \eqref{it:tz=0} for $\g$ from assertion~\eqref{it:X} of the theorem, and for each of them establish the corresponding necessary and sufficient condition for stability of a stationary point $X$, and separately, the claim of assertion~\eqref{it:exist}.

\medskip

In case~\eqref{it:tz>1} we have $d_\z > 1$. If $d_\z =4$, the algebra $\g$ is abelian; then all the points $X \in \g$ are stationary and stable by Lemma~\ref{l:lin}\eqref{it:cJskew}. We cannot have $d_\z=3$, and if $d_\z=2$, there is an orthonormal basis $\{e_i\}$ for $(\g, \ip)$ such that the only nonzero bracket is $[e_1,e_2]=\a e_3, \, \a \ne 0$. Relative to this basis, equation~\eqref{eq:de} takes the form $\dot{y}_1=-\a y_2 y_3,\,\dot{y}_2=\a y_1 y_3,\, \dot{y}_3=\dot{y}_4=0$ for $Y=\sum_{i=1}^4 y_ie_i$. A point $X$ is stationary if it either belongs to the centre $\z=\Span(e_3,e_4)$, or if $x_3=0$ and one of $x_1, x_2$ is nonzero. In the former case, $X$ is stable by Lemma~\ref{l:lin}\eqref{it:cJskew}, and in the latter, the trajectory of the solution starting at the point $X+\ve e_3$, with a small $\ve \ne 0$, is a circle which does not remain close to $X$, so that such $X$ is unstable. This proves assertion~\eqref{it:X} in case~\eqref{it:tz>1}. Assertion~\eqref{it:exist} in this case also follows: we always have nonzero stable stationary points (lying in the centre) and always have unstable ones (unless $\g$ is abelian).

\medskip

In case~\eqref{it:tz=0}, let $\{e_i\}$ be an orthonormal basis for $(\g, \ip)$ such that $\ag=\Span(e_2,e_3,e_4)$. Denote $A=(\ad_{e_1})_{|\ag}$ (note that if we choose any vector in $\g \setminus \ag$ other than $e_1$ as in the statement of the assertion, the resulting $A$ will only differ by multiplication by a nonzero constant). Note that $\Tr A = 0$ by unimodularity, and that $\det A \ne 0$, as $\z = 0$. We need the following simple facts.

\begin{lemma} \label{l:AinR3}
  In the above notation, we have the following.
  \begin{enumerate}[label=\emph{(\roman*)},ref=\roman*]
    \item \label{it:Ano0}
    No eigenvalue of $A$ has zero real part.

    \item \label{it:Acone}
    The cone $\<Ax,x\>=0$ in $\ag$ is not a single point.

    \item \label{it:kfinite}
    The function $\phi_x(s) = \|\exp(sA^t)x\|^2$ is non-constant for $x \in \ag \setminus \{0\}$, and attains a \emph{(}positive, global\emph{)} minimum for all $x$ in an open, dense subset of $\ag$.
  \end{enumerate}
\end{lemma}
\begin{proof}
  For assertion~\eqref{it:Ano0}, we note that if $A$ has a nonzero imaginary eigenvalue, then from $\Tr A = 0$, it also has a zero eigenvalue, which contradicts $\det A \ne 0$.

  To prove assertion~\eqref{it:Acone} it suffices to show that the function $x \mapsto \<Ax,x\>$ on $\ag$ takes values of both signs. If all three eigenvalues of $A$ are real, we take for $x$ eigenvectors of $A$ and use the facts that $\Tr A = 0$ and $\det A \ne 0$. If $A$ has non-real eigenvalues $\a + \ri \b$ with corresponding eigenvector $x+\ri y$, then the remaining eigenvalue is $-2\a$ with corresponding eigenvector $z$, where $x, y, z \in \ag$; then $\<Az,z\>=-2\a \|z\|^2$ and $\<Ax,x\>+\<Ay,y\>^2=\a(\|x\|^2+\|y\|^2)$ (note that $\a \ne 0$ by assertion~\eqref{it:Ano0}).

  For assertion~\eqref{it:kfinite}, we consider the decomposition of $\ag$ into the direct sum of the (real) Jordan subspaces $\ag_{\la_i}$ of $A$, where $\la_i=\a_i + \ri \b_i$ are the eigenvalues of $A$. Let $x=\sum_i x_{(i)},\; x_{(i)} \in \ag_{\la_i}$, be the corresponding (linear) decomposition of $x$. If $x_{(i)} \ne 0$ is a term in this decomposition with the maximal $|\a_i|$, then $\phi_x(s)$ grows at least as $e^{2s\a_i}\|x_{(i)}\|^2$ when $s\a_i \to \infty$. As $\a_i \ne 0$ by assertion~\eqref{it:Ano0}, this proves the first statement of assertion~\eqref{it:kfinite}. To prove the second statement we choose $x$ outside the union $\cup_i \ag_{\la_i}$ (note that $A$ cannot be a single Jordan cell by assertion~\eqref{it:Ano0} and as $\Tr A = 0$). Then by the above argument, $\phi_x(s) \to \infty$ for $s \to \pm \infty$, and the claim follows as $\phi_x(s) > 0$ for $x \ne 0$.
\end{proof}

Equation \eqref{eq:de} takes the form
\begin{equation}\label{eq:EAz0}
  \dot{y}_1= -\<Ay,y\>,\qquad \dot{y}=y_1 A^t y,
\end{equation}
for $Y=y_1e_1+y \in \g$, where $y \in \ag$. As $\det A \ne 0$, the set of stationary points is the union of the axis $\br e_1$ and the cone in $\ag$ given by $\<Ax,x\>=0, \; x \in \ag$ (note that this cone is nontrivial by Lemma~\ref{l:AinR3}\eqref{it:Acone}). The points $X=x_1e_1, \, x_1 \ne 0$, are unstable by Lemma~\ref{l:lin}\eqref{it:trtr}: we have $\cJ_X= \left(\begin{smallmatrix} 0 & 0 \\ 0 & A^t \end{smallmatrix}\right)$, and so not all eigenvalues of $\cJ_X$ (in fact, not a single one) have zero real part by Lemma~\ref{l:AinR3}\eqref{it:Ano0}.

Let $X = x \in \ag$ be a stationary point. The point $X=0$ is stable, so we can assume $x \ne 0$. The trajectory of the solution of \eqref{eq:EAz0} with the initial condition $Y(0)=y_{1}^0 e_1 + y^0, \, y^0 \in \ag \setminus \{0\}$, lies on the cylindrical surface $\cp(Y(0))$ given by $(w,s) \mapsto we_1 + \exp(sA^t)y^0$ (the coadjoint orbit of $Y(0)$). Taking $Y(0)=X$ and restricting the first integral $I_0(Y)=\|Y\|^2$ to the surface $\cp(X)$ we obtain the function $F(w,s)= w^2 + \phi_x(s)$. Consider the expansion $\phi_x(s) = \|x\|^2 + a_k s^k + o(s^k)$, where $k > 0$ and $a_k \ne 0$ (note that such a $k$ exists by Lemma~\ref{l:AinR3}\eqref{it:kfinite}). If $k$ is even and $a_k > 0$, then the equation $F(w,s)=F(0,0)$ has locally only one solution, $(w,s)=(0,0)$, and so the level hypersurface $I_0(Y)=I_0(X)$ has only one point in common with the cylinder $\cp(X)$ in a neighbourhood of $X$ which implies that $X$ is stable (see Section~\ref{ss:Lyapunov}). Suppose that either $k$ is odd, or $a_k < 0$. Then the solution to $F(w,s)=F(0,0)$ is locally given by $w=\pm \sqrt{-a_ks^k} + \dots$ which, outside the point $(w,s)=(0,0)$, is the union of open analytic curves $\gamma_j, \, j=1, \dots, 2m,\, m \in \{1,2\}$, whose closures contain $(0,0)$. Note that locally these curves contain no stationary points of \eqref{eq:EAz0}, as for all of them, we locally have $y_1 = w \ne 0$, and so each of them is a trajectory of a solution of \eqref{eq:EAz0} which tends to $X$ for either $t \to \infty$ or $t \to -\infty$. But both the system~\eqref{eq:EAz0} and the stationary point $X$ are invariant with respect to the change of variables $(y_1,t) \leftrightarrow (-y_1,-t)$, and so $m$ of these $2m$ trajectories tend to $X$ when $t \to \infty$, and another $m$, when $t \to -\infty$. This implies that $X$ is unstable which completes the proof of assertion~\eqref{it:X} in case~\eqref{it:tz=0}.

\begin{remark} \label{r:A^n}
  The coefficients $c_N$ of the Taylor expansion $\phi_x(s) = \sum_{N=0}^\infty c_N s^N$ are given by $c_N=\sum_{j=0}^N (j!(N-j)!)^{-1} \<A^{N-j} (A^t)^j x, x\>$, and so $c_0=\|x\|^2, \, c_1=2\<Ax,x\>=0$ (as $X$ is a stationary point), $c_2=\<(A^2+AA^t)x,x\>$, and so on. One can show (by a direct computation) that the condition of stability in assertion~\eqref{it:tz=0} is, in fact, finite: if $c_N=0$ for $N \le 5$, then $c_6$ is always positive. Such order of tangency between the sphere $I_0(Y)=I_0(X)$ and the coadjoint orbit may indeed occur: for example, for the matrix $A=\left(\begin{smallmatrix} 1 & 1 & 2 \\ -3 & -1 & 4 \\ -2 & 0 & 0  \end{smallmatrix}\right)$ and $x=(0,0,1)^t$, we have $\phi_x(s) = 1 + \frac85 s^6 + o(s^6)$.

  We also note that for a nonzero stationary point $X \in \ag$ one has $\cJ_X= \left(\begin{smallmatrix} 0 & -x^t (A+A^t) \\ A^tx & 0 \end{smallmatrix} \right)$, and so $\rk(\cJ_X) \le 2$ and $\sigma_2(\cJ_X)=c_2$, so the ``first" condition which may imply stability, $c_2 > 0$, is equivalent to $\sigma_2(\cJ_X) > 0$ (cf. Lemma~\ref{l:lin}\eqref{it:sigmas4}).
\end{remark}

For assertion~\eqref{it:exist} in case~\eqref{it:tz=0} we note that unstable stationary points always exist (nonzero multiples of $e_1$). To prove the existence of nonzero stable points, take an arbitrary $y \in \ag$ from the open, dense set as in Lemma~\ref{l:AinR3}\eqref{it:kfinite}. Then $\phi_y(s)$ attains a positive, global minimum at some point $s=s_0$; this point is an isolated point of the set $\{s: \phi_y(s)=\phi_y(s_0)\}$, as $\phi_y$ is analytic and non-constant. Denote $x=\exp(s_0A^t)y$ (note that $x \ne 0$) and $X=x$. Then $X \in \cp(Y)$, where $Y=y$, and the restriction of the first integral $I_0$ to the cylinder $\cp(Y)\,(=\cp(X))$ has a strict, isolated minimum at $X$; geometrically, $X$ is the first point of tangency of the family of expanding spheres centred at the origin with the coadjoint orbit $\cp(Y)$. This implies that $X$ is a nonzero stable stationary point of \eqref{eq:EAz0} (see Section~\ref{ss:Lyapunov}), which completes the proof of the theorem in case~\eqref{it:tz=0}.

There are clear similarities between the proof of this case and the proof of Theorem~\ref{t:3} in the non-unimodular case; this is further explored in a paper under preparation.

\medskip

In case~\eqref{it:tz=1}, choose an orthonormal basis $\{e_i\}$ for $(\g, \ip)$ such that $\z=\br e_4$, and denote $\m=\Span(e_1,e_2,e_3)$. Then the Lie brackets on $\g$ are given by
\begin{equation} \label{eq:brtz1}
  [e_4, x]=0, \qquad [x,y]=S(x \times y) + \<l, x \times y\>e_4, \qquad x, y \in \m,
\end{equation}
where $S$ is a symmetric operator on $\m, \, l \in \m$ and $\times$ is the cross product on $\m$ (the fact that $S$ is symmetric follows from unimodularity). To avoid having a bigger centre, we additionally require that $S \ne 0$, and, if $\rk S = 1$, that $l$ does not lie in the image of $S$.

Equation \eqref{eq:de} takes the form (for $Y=y+y_4e_4 \in \g$, where $y \in \m$)
\begin{equation}\label{eq:EAz1}
  \dot{y}=(Sy+y_4 l) \times y,\qquad \dot{y}_4= 0.
\end{equation}
The set of stationary points is the union of the axis $\br e_4$ (the centre) and the set of points
\begin{equation}\label{eq:statz1}
X=x+x_4e_4, \quad x \in \m \setminus \{0\}, \quad \text{such that } (S-\mu \,\id_\m)\,x=-x_4 l, \quad \text{for some } \mu \in \br
\end{equation}
(note that such a $\mu$ is uniquely determined by $x \ne 0$). The points $X=x_4e_4$ are stable by Lemma~\ref{l:lin}\eqref{it:cJskew}, and so we will be interested in the stationary points given by \eqref{eq:statz1}. We may again consider the restriction of $I_0$ to coadjoint orbits, as in the previous case, but it is easier to use the fact that the system~\eqref{eq:EAz1} has three explicit polynomial first integrals: $I_1(Y)=y_4, \; I_2(Y)=\<Sy,y\>+2y_4\<l,y\>$ and $I_0$ (where $Y=y+y_4e_4, \, y \in \m$), and so ``in principle", we know all the trajectories.

Let $X$ be a stationary point given by~\eqref{eq:EAz1} and let $\G_X=\{Y \in \g \, : \, I_i(Y)=I_i(X), \, i=0,1,2\}$. If $Y=X+U \in \G_X$, where $U=u+u_4e_4, \, u \in \m$, we obtain $u_4=0, \, \|u\|^2 + 2\<u,x\>=0$, and then by~\eqref{eq:EAz1}, $\<Su,u\>=-2\<Sx+2x_4l,u\>=-2\mu\<x,u\>=\mu \|u\|^2$. Denote $Q=S-\mu \, \id_\m$. Then $u \in \m$ belongs to the intersection of the ellipsoid $E_x$ given by $\|u\|^2 + 2\<u,x\>=0$ and the cone $C_x$ given by $\<Qu,u\>=0$. The tangent plane to $E_x$ at $x$ is $x^\perp$, and so if the restriction $\psi_x$ of the quadratic form $v \mapsto \<Qv,v\>$ to $x^\perp$ is strictly definite (either positive or negative), then $E_x \cap C_x$ locally consists of a single point $u=0$, and so for some neighbourhood $\mU(X)$ of $X$ we have $\G_X \cap \mU(X)=X$, which implies that $X$ is stable. This condition for $\psi_x$ can be expressed in terms of $\cJ_X$. From~\eqref{eq:EAz1} we have $\cJ_XY=(Qy+y_4 l) \times x$ for $Y=y+y_4e_4, \, y \in \m$, and so $\rk \cJ_X \le 2$ and moreover, $\sigma_2(\cJ_X) = \|x\|^2 \det \psi_x$ (by an easy calculation choosing a basis for $\m$ in such a way that $x$ is a multiple of $e_3$). We deduce that if $\sigma_2(\cJ_X) > 0$, then $X$ is stable (as per condition~\eqref{it:tz=1}\eqref{it:s2J>0}). By Lemma~\ref{l:lin}\eqref{it:sigmas4} it remains to consider the case $\sigma_2(\cJ_X) = 0$.

By Lemma~\ref{l:lin}\eqref{it:cJskew}, $X$ is stable when $\cJ_X=0$, so we suppose that $\sigma_2(\cJ_X) = 0$ and $\cJ_X \ne 0$. Specifying the basis $e_1, e_2, e_3$ for $\m$ in such a way that $x=\a e_3,\, \a \ne 0$, and that $e_1$ lies in the kernel of the matrix of $\psi_x$ we obtain
\begin{equation} \label{eq:QcJ}
  Q=\begin{pmatrix} 0 & 0 & q_{13} \\ 0 & q_{22} & q_{23} \\ q_{13} & q_{23} & q_{33} \end{pmatrix}, \quad
  \cJ_X=\a \begin{pmatrix} 0 & q_{22} & q_{23} & l_2\\ 0 & 0 & -q_{13} & - l_1\\ 0 & 0 & 0 & 0 \\ 0 & 0 & 0 & 0 \end{pmatrix}, \quad \text{with \ }\a Q e_3 = -x_4 l,
\end{equation}
where the last equation follows from~\eqref{eq:statz1}.

Suppose $\det Q = 0$ and $\sigma_2(Q)>0$. Then the cone $C_x$ given by $\<Qu,u\>=0$ is a line, and so its intersection with $E_x$ is locally a single point $u=0$, which implies that $X$ is stable. These conditions on $Q$ are equivalent to conditions on $\cJ_X$ given in \eqref{it:tz=1}\eqref{it:Jnilp}. Indeed, we have $\cJ_{e_4}Y=(l \times y, 0)^t$, and so from~\eqref{eq:QcJ}, $\<X,e_4\> \Tr (\cJ_X \cJ_{e_4})=-\a^2\sigma_2(Q)$. If $\det Q = 0$ and $\sigma_2(Q)>0$, then $q_{13}=0, \, q_{33} \ne 0$, and so from the last equation of~\eqref{eq:QcJ} we obtain $l_1=0$, so $\rk \cJ_X = 1$. Conversely, if $\<X,e_4\> \Tr (\cJ_X \cJ_{e_4}) < 0$, then $\sigma_2(Q) > 0$, and so $q_{22} \ne 0$ which by $\rk \cJ_X=1$ implies $q_{13}=l_1=0$ from \eqref{eq:QcJ}.

It remains to show that in all the other cases, the point $X$ is unstable. Suppose $Y=X+U \in \G_X$ is a stationary point. Then from the above we have $u_4=0, \; \<Qu,u\>=0$ and $2\<x,u\>+\|u\|^2=0$, and $Sy+x_4l=\mu' y$ for some $\mu' \in \br$, which gives $Qu=(\mu'-\mu)(x+u)$, and then $0=(\mu'-\mu)\<x+u,u\>=\frac12(\mu'-\mu)\|u\|^2$. If $u=0$ we get $Y=X$, and so $\G_X$ contains stationary points $Y=X+U$ different from $X$ only when $Qu = 0$, with $u \ne 0$ and $2\<x,u\>+\|u\|^2=0$.

Now if $\det Q \ne 0$, then from~\eqref{eq:QcJ}, one of the generatrices of the cone $C_x$, the axis $\br e_1$, lies in the tangent plane $u_3=0$ to the ellipsoid $E_x$ at $u=0$, and all the other generatrices intersect $E_x$ at $y=0$ and at one other point, so that $C_x \cap E_x$ (and hence $\G_X$) is homeomorphic to a circle. From the previous paragraph, $X$ is the only stationary point on $\G_X$, and so $\G_X \setminus \{X\}$ is a single trajectory, which implies that $X$ is unstable.

Next suppose that $\det Q = 0$ and $\sigma_2(Q) \le 0$. If $q_{22}=0$ in \eqref{eq:QcJ}, we can further specify $e_1, e_2$ to make $q_{13}$ zero. So we can always assume that $q_{13}=0$, and then $\sigma_2(Q) =q_{22}q_{33}-q_{23}^2 \le 0$. If this inequality is strict, then the cone $C_x$ is the union of two planes, and so $\G_X$ is either a circle passing through $X$ or a union of two circles which touch at $X$. As the equations $Qu=0,\, 2\<x,u\>+\|u\|^2=0$ imply $u=0$, $\G_X$ contains no stationary points other than $X$, and so every connected component of $\G_X \setminus \{X\}$ is a single trajectory, which again implies that $X$ is unstable.

The last case to consider is when $q_{13} = 0$ and $\sigma_2(Q) = q_{22}q_{33}-q_{23}^2=0$. If $q_{22}=0$, then $q_{23}=0$ and so $x_4 l = - \a (0,0,q_{33})^t$ by~\eqref{eq:QcJ}. As $\cJ_X \ne 0$, at least one of $l_1, l_2$ must be nonzero, and so we get $x_4=0$ and $Q=0$. Then~\eqref{eq:EAz1} takes the form $\dot{y}= y_4 l \times y,\; \dot{y}_4= 0$, and the trajectory of its solution starting at the point $Y = \a e_3 + \ve e_4$ close to $X = \a e_3$ is a circle which does not remain close to $X$; so $X$ is unstable. We can therefore assume that $q_{22} \ne 0$, and so $Qy= c (\b y_2 + \gamma y_3)(0,\b,\gamma)^t$ for some $\b, \gamma \in \br, \, c=\pm1, \, \b \ne 0$. From~\eqref{eq:QcJ} we obtain $x_4 l = - \a c \gamma (0,\b,\gamma)^t$. If $x_4=0$, then $\gamma=0$ and the trajectory of the solution to~\eqref{eq:EAz1} starting at the point $Y = \ve e_2 + \a e_3$ close to $X = \a e_3$ is a circle which does not remain close to $X$, which implies that $X$ is unstable. We can therefore assume that $x_4 \ne 0$, and so  $l = - \a c \gamma x_4^{-1} (0,\b,\gamma)^t$. Then~\eqref{eq:EAz1} takes the form  $\dot{y}= c (\b y_2 + \gamma y_3 - \a \gamma x_4^{-1} y_4) (0,\b,\gamma)^t \times y,\; \dot{y}_4= 0$. It has a first integral $I_3(Y)= \b y_2 + \gamma y_3$. For a small $\ve \ne 0$, the set $C'$ of points $Y$ such that $I_0(Y)=I_0(X), \, I_1(Y)=I_1(X), \, I_3(Y)=I_3(X)+\ve$ is a circle. Stationary points lying on $C'$ must satisfy $(0, \b, \gamma)^t \times y =0$, and an easy calculation shows that $C'$ contains no stationary points, unless $\ve^2 +2 \a \gamma \ve -\a^2\b^2=0$ which we can always avoid by a choice of $\ve$. It follows that $C'$ is the trajectory of a single periodic solution; it passes close to $X$, but does not remain close $X$, and so $X$ is unstable.

This proves assertion~\eqref{it:X} in case~\eqref{it:tz=1}.

For assertion~\eqref{it:exist} in case~\eqref{it:tz=1}, we note that nonzero stable stationary points always exist (nonzero multiples of $e_4$, by Lemma~\ref{l:lin}\eqref{it:cJskew}).

To prove the existence of unstable points, we choose a basis $\{e_1,e_2,e_3\}$ of eigenvectors of $S$ for $\m$. We can assume that the corresponding eigenvalues $\la_i$ satisfy $\la_1 \ge \la_2 \ge \la_3$. Then by \eqref{eq:statz1}, a point $X=x+x_4e_4, \, x \in \m, \, x \ne 0$, is stationary if for some $\mu \in \br$ it satisfies the equations $(\la_i - \mu) x_i = -x_4 l_i$ for $i=1,2,3$. At such a point $X$ we have $\sigma_2(\cJ_X)= (\la_1-\mu)(\la_2-\mu)x_3^2 + (\la_3-\mu)(\la_1-\mu)x_2^2 + (\la_2-\mu)(\la_3-\mu)x_1^2$. If $\la_1 > \la_2 > \la_3$, we take $X=e_2, \, \mu = \la_2$. Then $\sigma_2(\cJ_X) < 0$ and so $X$ is unstable by~\eqref{it:X}. Suppose $\la_1 = \la_2 > \la_3$; the case $\la_1 > \la_2 = \la_3$ reduces to this one by changing the sign of $S$ (changing $e_3$ to $-e_3$). We take $x_4=1$ and $x_i = -(\la_i - \mu)^{-1} x_4 l_i$ for $i=1,2,3$, with $\mu \ne \la_1,\la_3$. Then $\sigma_2(\cJ_X)= (l_1^2+l_2^2)a + l_3^2 a^{-2}$, where $a=(\la_1-\mu)^{-1}(\la_3-\mu)$. If at least one of $l_1, l_2$ is nonzero, we have $\sigma_2(\cJ_X) < 0$ when $a \to -\infty$ (that is, when $\mu \to \la_1-$), so such stationary points $X$ are unstable. Suppose $l_1=l_2=0$. Take $\mu=\la_1, \, x_3=x_4=0$, and $x_1, x_2$ nonzero. We have $\cJ_X=(-x_2,x_1,0,0)^t (0,0,\la_3-\la_1,l_3)$, so that $\rk \cJ_X=0$ and $\sigma_2(\cJ_X)=0$. Then $\cJ_{e_4}Y = l_3 e_3 \times y$, and so $\Tr (\cJ_X \cJ_Z) = 0$ which implies that such points $X$ are unstable by~\eqref{it:tz=1}\eqref{it:Jnilp}.

It remains to consider the case when $\la_1 = \la_2 = \la_3$. Then~\eqref{eq:EAz1} takes the form $\dot{y} = y_4 l \times y,\, \dot{y_4} = 0$. If $l = 0$, all the points are trivially stable. If $l \ne 0$, any point $X =x \in \m$ is stationary; taking $x$ non-proportional to $l$ we get that $X$ is unstable, as the trajectory of the solution starting at $Y=X+\ve e_4$, with a small positive $\ve$, is a circle which does not remain close to $X$.

So we always have unstable points, unless $S$ is scalar and $l=0$, which by \eqref{eq:brtz1} gives that either $\g$ is abelian or $\g = \so(3) \oplus \br$ and $\ip$ is bi-invariant inner product.


\end{document}